\documentclass[12pt,a4paper]{article}

\usepackage[dvips]{graphicx}
\usepackage{calc}
\usepackage{color}
\usepackage{amsmath}
\usepackage{amssymb}
\usepackage{amscd}
\usepackage{amsthm}
\usepackage{amsbsy}
\usepackage{delarray}
\usepackage{enumerate}
\usepackage[T1]{fontenc}
\usepackage{inputenc}
\usepackage{enumerate}

\begin{document}



\setlength{\parindent}{5mm}
\renewcommand{\leq}{\leqslant}
\renewcommand{\geq}{\geqslant}
\newcommand{\N}{\mathbb{N}}
\newcommand{\sph}{\mathbb{S}}
\newcommand{\Z}{\mathbb{Z}}
\newcommand{\R}{\mathbb{R}}
\newcommand{\C}{\mathbb{C}}
\newcommand{\F}{\mathbb{F}}
\newcommand{\g}{\mathfrak{g}}
\newcommand{\h}{\mathfrak{h}}
\newcommand{\K}{\mathbb{K}}
\newcommand{\RN}{\mathbb{R}^{2n}}
\newcommand{\ci}{c^{\infty}}
\newcommand{\derive}[2]{\frac{\partial{#1}}{\partial{#2}}}
\renewcommand{\S}{\mathbb{S}}
\renewcommand{\H}{\mathbb{H}}
\newcommand{\eps}{\varepsilon}

\theoremstyle{plain}
\newtheorem{theo}{Theorem}[subsection]
\newtheorem{prop}[theo]{Proposition}
\newtheorem{lemma}[theo]{Lemma}
\newtheorem{definition}[theo]{Definition}
\newtheorem*{notation*}{Notation}
\newtheorem*{notations*}{Notations}
\newtheorem{corol}[theo]{Corollaire}
\newtheorem{conj}[theo]{Conjecture}
\newtheorem{question}[theo]{Question}
\newtheorem*{question*}{Question}
\newenvironment{demo}[1][]{\addvspace{8mm} \emph{Proof #1.
    ---~~}}{~~~$\Box$\bigskip}

\newlength{\espaceavantspecialthm}
\newlength{\espaceapresspecialthm}
\setlength{\espaceavantspecialthm}{\topsep} \setlength{\espaceapresspecialthm}{\topsep}

\newenvironment{example}[1][]{\refstepcounter{theo} 
\vskip \espaceavantspecialthm \noindent \textsc{Example~\thetheo
#1.} }%
{\vskip \espaceapresspecialthm}

\newenvironment{remark}[1][]{\refstepcounter{theo} 
\vskip \espaceavantspecialthm \noindent \textsc{Remark~\thetheo
#1.} }%
{\vskip \espaceapresspecialthm}

\def\Homeo{\mathrm{Homeo}}
\def\Hameo{\mathrm{Hameo}}
\def\Diffeo{\mathrm{Diffeo}}
\def\Symp{\mathrm{Symp}}
\def\Id{\mathrm{Id}}
\newcommand{\norm}[1]{||#1||}
\def\Ham{\mathrm{Ham}}
\def\Cal{\mathrm{Cal}}

\title{The Calabi invariant for some groups of homeomorphisms}
\author{Vincent Humili\`ere$^1$}
\normalsize \maketitle \footnotetext[1]{LMU Munich, Germany (Formerly in Ecole Polytechnique, France)

Supported also by the ANR project "Symplexe".

Email: \texttt{vincent.humiliere@mathematik.uni-muenchen.de}}

\abstract{We show that the Calabi homomorphism extends to some groups of homeomorphisms on exact symplectic manifolds.

The proof is based on the uniqueness of the generating Hamiltonian (proved by Viterbo) of continuous Hamiltonian
isotopies (introduced by Müller and Oh).}

\section{Introduction}
\subsection{The Calabi homomorphism}

Let $(M,\omega)$ be a symplectic manifold, supposed to be \emph{exact}, that is $\omega=d\lambda$ for some 1-form
$\lambda$ called \emph{Liouville form}. Equivalently, this also means that there exists a vector field $X$ such that
the Lie derivative satisfies: $\mathcal{L}_X\omega=\omega$. The vector field $X$ is called the \emph{Liouville vector
field} and is related to the 1-form $\lambda$ by the relation $\iota_X\omega=\lambda$. For instance, cotangent bundles
are exact symplectic manifolds.

Thanks to the work of Banyaga \cite{banyaga,bounemoura}, the algebraic structure of the group $\Ham_c(M,\omega)$ of
smooth compactly supported Hamiltonian diffeomorphisms of $(M,\omega)$ is quite well understood: there exists a group
homomorphism, defined by Calabi \cite{calabi}
$$\Cal:\Ham_c(M,\omega)\to\R,$$ whose kernel $\ker(\Cal)$ is a simple group.

The Calabi homomorphism is defined as follows. Let $\phi\in\Ham_c(M,\omega)$ and let $H$ be a compactly supported
Hamiltonian function generating $\phi$, i.e., a smooth function $[0,1]\times M\to\R$ such that: \begin{itemize}
\item $\phi$ is the time one map of the flow $(\phi_H^t)_{t\in[0,1]}$ of the only time dependent vector field $X_H$
satisfying at any time $t\in[0,1]$, $$\iota_{X_H(t,\cdot)}\omega=dH(t,\cdot),$$
\item there exists a compact set in $M$ that contains all the supports of the functions $H_t=H(t,\cdot)$, for
$t\in[0,1]$.
\end{itemize}
Then, by definition, \begin{equation}\label{formule Calabi}\Cal(\phi)=\int_0^1\int_MH(t,x)\omega^ddt,\end{equation}
where $d$ is half the dimension of $M$. This expression does not depend on the choice of the generating function $H$,
and gives a group homomorphism.

\subsection{Question and result}

We consider the following question.

\begin{question}\label{question extension calabi}To which groups of homeomorphisms 
does the Calabi homomorphism extend?
\end{question}

Note that the Calabi homomorphism does not behave continuously with respect to the $C^0$-topology, as shows the
following example.

\begin{example}Let $\phi\in\Ham_c(\R^2,rdr\wedge d\theta)$, and consider the sequence $(\phi_n)$ in
$\Ham_c(\R^2,rdr\wedge d\theta)$ given by
$$\phi_n(r,\theta)=\frac1n\phi^{4n}(nr,\theta).$$
This sequence converges in the $C^0$-sense to $\Id$, but one can easily check that its Calabi invariant remains
constant.
\end{example}

We will define three interesting groups of homeomorphisms, denoted $G_1$, $G_2$ and $G_3$, and prove the following
result.

\begin{theo}\label{theorem principal}The Calabi homomorphism extends to a group homomorphism $G_3\to\R$. Moreover, we
have the following inclusions $\Ham_c(M,\omega)\subset G_1\subset G_2\subset G_3$.
\end{theo}

We will give the full definitions of $G_1$, $G_2$ and $G_3$ in Section \ref{section three groups}. Let us still give
here an idea of what they are:
\begin{description}
  \item[$G_1$] is the identity component of the group of compactly supported symplectic bilipschitz homeomorphisms
  whose flux is zero (see Section \ref{section G1}).
  \item[$G_2$] is the group generated by the homeomorphisms that admit particular generating functions (see Section \ref{section G2}).
  \item[$G_3$] is the group of homeomorphisms $\phi$ such that (on some interval where it is well defined) the isotopy
  $t\mapsto [\mu_t,\phi]$ is a $C^0$-Hamiltonian isotopy (in the sense of \cite{Oh}). Here, $\mu_t$ denotes the flow
  generated by the Liouville vector field $X$, and $[\mu_t,\phi]=\mu_t\circ\phi\circ\mu_t^{-1}\circ\phi^{-1}$ (see Section \ref{section
  G3}).
\end{description}

\medskip
\begin{remark}In the special case of the (2-dimensional) open disk, the fact that the Calabi homomorphism extends to $G_1$ was
already proved by Haissinsky \cite{haissinsky}\footnote{Area preserving quasiconformal maps of the plane are
bilipschitz. Therefore, Haissinsky's result is precisely the fact that the Calabi homomorphism extends to $G_1$.}. His
methods are completely different.

Let us also mention that Gambaudo and Ghys have proved that two diffeomorphisms of the disk that are conjugated by an
area preserving homeomorphism have same Calabi invariant \cite{gamb-ghys}.
\end{remark}

\subsection{Motivation}

Our motivation for this work comes from two distinct problems. The first one comes from the following question which
remains open.

\begin{question}[Fathi \cite{fathi}]\label{question simplicité}Is the group $\Homeo_c(\mathbb{D}_2,area)$ of compactly supported area preserving homeomorphisms
of the disk a simple group ?
\end{question}

Several non-trivial normal subgroups of $\Homeo_c(\mathbb{D}_2,area)$ have been defined by Ghys \cite{bounemoura},
Müller-Oh \cite{muller-oh} and recently by Le Roux \cite{leroux}. But so far, no one has been able to prove that any of
them is a proper subgroup.

Our study is inspired by the work of Müller and Oh. They introduced on any symplectic manifold $(M,\omega)$ a group
denoted $\Hameo(M,\omega)$, whose elements are homeomorphisms called \emph{hameomorphisms} (as the contraction of
"Hamiltonian homeomorphisms"). This group contains all compactly supported Hamiltonian diffeomorphisms and, in the case
of the disk, forms a normal subgroup of $\Homeo_c(\mathbb{D}_2,area)$. A. Fathi noticed that if one could extend the
Calabi homomorphism to the group of hameomorphisms, then it would be necessarily a proper subgroup, and
$\Homeo_c(\mathbb{D}_2,area)$ would not be simple.

In the present paper, we propose a different approach: instead of constructing a group which is known to be normal but
on which it is unknown whether the Calabi homomorphism extends, we construct a group to which the Calabi invariant
extends but for which it is unknown whether it is normal.

\bigskip Another motivation is a very natural general problem: how can one generalize Hamiltonian dynamics in a
non-smooth context? or (less optimistic) which properties of Hamiltonian maps can be extended? The present paper
concentrates on a particular aspect: the Calabi homomorphism.

Our interest in the groups $G_1$ and $G_2$ comes from the fact that they give large families of examples of elements of
$G_3$, but also from the fact they are quite natural generalizations of the Hamiltonian group, which could be
considered to study the extension of other aspects of Hamiltonian dynamics. As an example, all the rigidity results
obtained on Hamiltonian diffeomorphisms using generating functions technics may also hold for the elements of $G_2$
(and thus of $G_1$).

Several other possible groups generalizing the Hamiltonian group have already been considered in literature. The group
$\Hameo(M,\omega)$ mentioned above is one of them, another has been studied by the author in \cite{humi}. But this
direction of research is still to be developed.

\section{The three groups}\label{section three groups}
\subsection{The group $G_1$}\label{section G1}

The group $\Ham_c(M,\omega)$ can be characterised as the set of all symplectic diffeomorphism which are compactly
supported, isotopic to the identity and with zero flux. It is thus natural to introduce the following definition.

\begin{definition}\label{definition G1}We denote by $G_1$ the identity component of the group of compactly supported
bilipschitz (for some Riemannian metric) symplectic homeomorphisms whose flux is zero.
\end{definition}

\begin{remark}Since Lipschitz maps are almost everywhere differentiable, the pull-back of a differential form by a bilipschitz map is
well-defined as a differentiable form with $L^{\infty}$ coefficients. Therefore, as in the smooth case, a bilipschitz
homeomorphism $\phi$ of $M$ is symplectic if $\phi^*\omega=\omega$, and has zero flux if the one form
$\lambda-\phi^*\lambda$ is exact (recall that $(M,\omega)$ is supposed to be exact, with Liouville form $\lambda$).

Note that a bilipschitz homeomorphism which is the $C^0$-limit of smooth symplectomorphisms is symplectic in this
sense.
\end{remark}

\subsection{The group $G_2$}\label{section G2}

The group $G_2$ is another natural generalisation of the Hamiltonian group, based on the notion of generating function,
that we describe now.

\medskip
First recall that, according to Weinstein's Neighbourhood Theorem \cite{Weinstein}, there exists a neighbourhood $U$ of
the diagonal $\Delta$ in the symplectic manifold $(M\times M, \omega\oplus(-\omega))$, symplectomorphic to a
neighbourhood of the zero section in the cotangent $T^*\Delta$ and hence to a neighbourhood $V$ of the zero section in
$T^*M$. We will denote $j:U\to V$ such a symplectomorphism, and call it a \textit{Weinstein chart}.

Now, for any symplectic diffeomorphism $\phi$ of $M$, the image of its graph
$$L_{\phi}=j(\text{graph}(\phi))=j(\{(x,\phi(x))\in M\times M\,|\,x\in M\})$$ is a lagrangian submanifold in $T^*M$. Moreover, $L_{\phi}$ is exact
if and only if $\phi$ is Hamiltonian. If in addition $\phi$ is sufficiently close to the identity in the $C^1$-sense,
$L_{\phi}$ is the graph of the differential of a smooth compactly supported function $S:M\to\R$:
$$L_{\phi}=\text{graph}(dS)=\{(x,dS(x))\in T^*M\,|\,x\in M\}.$$ We then say that $\phi$ admits $S$ as a generating function.

Since any Hamiltonian diffeomorphism can be written as a product of Hamiltonian diffeomorphisms $C^1$-close to the
identity, the group $Ham_c(M,\omega)$ can be characterized as the subgroup of the diffeomorphisms of $M$ generated by
the elements that admit smooth compactly supported generating functions. When one tries to extend a property of
Hamiltonian diffeomorphisms to homeomorphisms (the Calabi invariant in our case), it is thus natural to consider
homeomorphisms that admits generating functions. This idea leads to the following definition.

\begin{definition}\label{defintion G2}A $C^1$ compactly supported function $S:M\to\R$ is called an \textnormal{admissible generating
function} if there exist a homeomorphism $\phi$ of $M$, and a symplectic diffeomorphism $j$ between a neighbourhood $U$
of the diagonal in $M\times M$ and a neighbourhood $V$ of the zero section in $T^*M$ such that
\begin{itemize}
\item $\text{graph}(\phi)\subset U$,
\item $\text{graph}(dS)\subset V$,
\item $j(\text{graph}(\phi))=\text{graph}(dS)$.
\end{itemize}
The homeomorphism $\phi$ associated to $S$ is called an \textnormal{admissible homeomorphism}.

An admissible generating function is called \textnormal{super-admissible} if it is the limit in the $C^1$-sense of a
sequence of $C^{\infty}$ admissible generating functions. A \textnormal{super-admissible homeomorphism} is an
admissible homeomorphism associated to a super-admissible generating function.

We denote by $G_2$ the group generated by homeomorphisms $\phi$ for which there exists some real number $\delta>0$ such
that for any $t$ in $[0,\delta]$, the conjugation by the Liouville flow $\mu_t\circ\phi\circ\mu_t^{-1}$ is also
super-admissible.
\end{definition}

\begin{remark}\label{remarque def conjugaison}As in the introduction, $\mu_t(x)$ denotes the flow (when it is defined) of the Liouville vector field $X$, at
time $t$ and point $x\in M$. Note that it satisfies $\mu_t^*\omega=e^t\omega$.

Let $\phi$ be a compactly supported homeomorphism of $M$. Then there exists a real number $\delta>0$, such that for any
$t\in[0,\delta]$, $\mu^t$ and $(\mu^t)^{-1}$ are well defined on the support of $\phi$. Thus, the conjugation
$\mu_t\circ\phi\circ\mu_t^{-1}$ is well defined on $\mu_t(\text{Supp}(\phi))$. In the complement of this set, it is the
identity where it is defined. Therefore, we can extend it to a well defined homeomorphism still denoted
$\mu_t\circ\phi\circ\mu_t^{-1}$ just by setting it to equal the identity where it is not defined.
\end{remark}

Even though the definition of $G_2$ looks quite strange, it is quite a large group. Indeed, it contains the group
$G_1$, as stated in Theorem \ref{theorem principal}, and it also contains a large family of examples, that we shall
construct in section \ref{section examples}.

\subsection{The group $G_3$}\label{section G3}

To define the group $G_3$ we first need the following notion.

\begin{definition}[Müller-Oh \cite{muller-oh}]A $C^0$-\textnormal{Hamiltonian isotopy} is a path $(\phi^t)_{t\in[0,\delta]}$ of homeomorphisms of $M$
for which there exist a compact set $K$ and a sequence of smooth Hamiltonian functions $H_n$ on $M$ with support in
$K$, such that
\begin{itemize}
\item $(H_n)$ converges to some continuous function $H:[0,\delta]\times M\to \R$ in the $C^0$-sense,
\item $(\phi_{H_n}^t)$ converges to $\phi^t$ in the $C^0$-sense, uniformly in $t\in[0,\delta]$.
\end{itemize}
The function $H$ is called a $C^0$-\textnormal{Hamiltonian function} generating $(\phi^t)$.
\end{definition}

\begin{remark}\label{remark hameotopy}The elements of $C^0$-Hamiltonian isotopies are \textit{symplectic homeomorphisms},
i.e., homeomorphisms which are the $C^0$ limit of a sequence of symplectic diffeomorphisms supported in a common
compact set.

It is not difficult to check that if $(\phi^t)$ and $(\psi^t)$ are two $C^0$-Hamiltonian isotopies generated by $F$ and
$G$, then $((\phi^t)^{-1})$ and $(\phi^t\circ\psi^t)$ are $C^0$-Hamiltonian isotopies generated by
$-F(t,(\phi^t)^{-1}(x))$ and $F(t,x)+G(t,\phi^t(x))$, and that if $f$ is any symplectic homeomorphism,
$(f^{-1}\circ\phi^t\circ f)$ is a $C^0$-Hamiltonian isotopy generated by $F(t,f(x))$. This means that the computations
are the same as in the smooth case.
\end{remark}

The main result concerning $C^0$-Hamiltonian isotopies is:

\begin{theo}[Viterbo \cite{viterbo}]\label{theorem viterbo}A given $C^0$-Hamiltonian isotopy is generated by a unique $C^0$-Hamiltonian
function.
\end{theo}

This theorem is the only non-trivial result needed in this paper. Its proof needs at some point a (hard!) rigidity
result in symplectic topology due to Gromov.

By Remark \ref{remarque def conjugaison}, for any compactly supported homeomorphism $\phi$ the commutator
$$[\mu_t,\phi]=\mu_t\circ\phi\circ\mu_t^{-1}\circ\phi^{-1},$$ is well defined, for $t$ small enough.

\begin{definition}\label{definition G3}We denote by $G_3$ the set of all compactly supported symplectic homeomorphisms $\phi$ for
which there exists some $\delta>0$ small
enough, such that the isotopy $([\mu_t,\phi])_{t\in[0,\delta]}$ is a $C^0$-Hamiltonian isotopy.
\end{definition}

Clearly, $G_3$ contains $Ham_c(M,\omega)$.

\begin{prop}\label{prop G3 groupe} The set $G_3$ is a group. Moreover, if the first compactly supported cohomology group $H^1_c(M,\R)$ vanishes,
$G_3$ does not depend on the choice of the Liouville vector field.
\end{prop}

\begin{demo} Let $\phi,\psi\in G_3$. For $\delta$ small enough $([\mu_t,\phi])_{t\in[0,\delta]}$ and
$([\mu_t,\psi])_{t\in[0,\delta]}$ are $C^0$-Hamiltonian isotopies. Then, note that
$$[\mu_t,\phi\circ\psi]=[\mu_t,\phi]\circ(\phi\circ[\mu_t,\psi]\circ\phi^{-1}),$$
and $$[\mu_t,\phi^{-1}]=\phi^{-1}\circ[\mu_t,\phi]^{-1}\circ\phi.$$ We conclude with Remark \ref{remark hameotopy} that
$G_3$ is a group.

Suppose now that $H^1_c(M,\R)=0$, and that $\mu_t'$ is the flow of another Liouville vector field. Then,
$\eta_t=\mu_t'\circ\mu_t^{-1}$ is a smooth symplectic isotopy which is Hamiltonian since $H^1_c(M,\R)=0$. Using once
again Remark \ref{remark hameotopy} and the identity
\begin{equation}\label{equation pas de H1}[\mu_t',\phi]=\eta_t\circ[\mu_t,\phi]\circ(\phi\circ\eta_t^{-1}\circ\phi^{-1}),\end{equation} we conclude that $G_3$ would be
the same if it was defined with another Liouville vector field.
\end{demo}

\section{Proof of the main theorem}
\subsection{Extension of the Calabi homomorphism}

In this section, we prove that the Calabi homomorphism extends to $G_3$. Let us first give a new formula for the
Calabi, for which we need to choose a Liouville form instead of choosing an isotopy.

\begin{lemma}\label{lemme new formula}Let $\phi\in\Ham_c(M,\omega)$ and let $H_{\lambda,\phi}$ be the generating
Hamiltonian function of the smooth Hamiltonian isotopy $([\mu_t,\phi])$. Then,
\begin{equation*} \Cal(\phi)=\frac{1}{d+1}\int_MH_{\lambda,\phi}(0,x)\omega^n.\end{equation*}
\end{lemma}

\begin{proof}First note that if $\phi$ is the time one map of some Hamiltonian function $H$, and if we suppose
$\mu_{\delta}^{-1}\circ\phi\circ\mu_{\delta}$ to be well defined, then it can be generated by the Hamiltonian function
$e^{\delta}H\circ\mu_{\delta}^{-1}$. After an easy change of variables in Equation (\ref{formule Calabi}), one gets
$$\Cal(\mu_{\delta}^{-1}\circ\phi\circ\mu_{\delta})=e^{(d+1)\delta}\Cal(\phi),$$
where $d$ is half the dimension of $M$. Thus,
$$\Cal([\mu_{\delta},\phi])=(e^{(d+1)\delta}-1)\Cal(\phi).$$
Hence, applying formula (\ref{formule Calabi}) to $H_{\lambda,\phi}$,
$$\Cal(\phi)=\frac{1}{e^{(d+1)\delta}-1}\int_0^\delta\int_MH_{\lambda,\phi}(t,x)\omega^n
dt.$$ Now, letting $\delta$ converge to $0$, we get the desired formula.
\end{proof}

Once this formula obtained, extending the Calabi homomorphism to $G_3$ is very easy, even though it relies on the "hard
symplectic topology" uniqueness Theorem \ref{theorem viterbo}.

\begin{proof}let $\phi\in G_3$ and let $H$ be the
\textbf{unique} $C^0$-Hamiltonian function generating $([\mu_t,\phi])_{t\in[0,\delta]}$ for some small $\delta$. We
set:
\begin{equation*}\widetilde{\Cal}(\phi)=\frac{1}{d+1}\int_MH(0,x)\omega^n.\end{equation*}
By Lemma \ref{lemme new formula}, $\widetilde{\Cal}$ coincide with $\Cal$ on $Ham_c(M,\omega)$. Moreover using Remark
\ref{remark hameotopy} and the formulas in the proof of Proposition \ref{prop G3 groupe}, one checks easily that
$\widetilde{\Cal}:G_3\to\R$ is a group homomorphism.
\end{proof}

\begin{remark} If $H^1_c(M,\R)=0$, then $\widetilde{\Cal}$ does not depend on the choice of the Liouville vector field.
This is an immediate consequence of Equation \ref{equation pas de H1}.\end{remark}

\subsection{Proof of the inclusion $G_1\subset G_2$}

We are going to prove that an element of $G_1$ which is sufficiently close to the identity in the bilipschitz sense is
a super-admissible homeomorphism (Definition \ref{defintion G2}). Since any element $g$ of $G_1$ can be written as a
product of elements of $G_1$ close to the identity (simply cut any path joining $g$ to the identity in small pieces),
this will imply that $G_1$ is included in $G_2$. This fact is standard for diffeomorphisms, and is not more difficult
in the bilipschitz case.

Let $g\in G_1$, close enough to the identity in the bilipschitz sense. Then, in particular $g$ is $C^0$-close to the
identity and its graph lies in the domain of a Weinstein chart $j:U\to V$. Now, the map $\Id\times g: M\to M\times M$,
$x\mapsto(x,g(x))$ is Lipschitz close to the diagonal inclusion $x\mapsto(x,x)$. As a consequence, the conjugated map $
a=j\circ(\Id\times g)\circ j^{-1}$ is Lipschitz-close to the zero section of the cotangent bundle $T^*M$. Standard
arguments (the same as in the $C^1$ case) then show that the image of $a$ is the graph of the section $s$ of $T^*M$
given by $$s=a\circ(\pi\circ a)^{-1},$$ where $\pi:T^*M\to M$ is the canonical projection. Moreover, this section $s$
is Lipschitz-close to the zero section.

%
%

It remains to prove that the Lipschitz 1-form $s$ is exact. This follows from the fact that the flux of the
homeomorphism $g$ vanishes. Indeed, since $g$ has zero flux, for any Liouville form $\lambda$, $(\Id\times
g)^*(\lambda\oplus(-\lambda))$ is an exact one form on $M$. Since the map $\Id\times g$ is homotopic to the map
$\Id\times\Id$, this implies that the pull-back of any primitive of $\omega\oplus(-\omega)$ is exact. Let $\lambda_0$
denotes the standard Liouville form on $T^*M$, one has $dj^*\lambda_0=(\omega\oplus(-\omega))$ hence $(\Id\times
g)^*j^*\lambda_0$ is exact. It follows that $s=s^*\lambda_0=(q^{-1})^* a^*\lambda_0$ is exact.

Now, if we denote by $S$ the compactly supported primitive of $s$, it is a $C^{1,1}$-function which is admissible by
construction. Moreover, it is small in the $C^{1,1}$ sense and thus can be approximated in the $C^1$ sense by
$C^2$-small smooth functions. But it is well known that $C^2$-small smooth functions are admissible. Therefore, $S$ is
super-admissible.

Finally, for $t$ small enough, $\mu_t\circ g\circ\mu_t^{-1}$ remains Lipschitz-close to the identity. Thus, $g$ is one
of the generators of $G_2$.$\quad\Box$

\subsection{Proof of the inclusion $G_2\subset G_3$}

Theorem \ref{theorem principal} clearly follows from the following proposition. We denote by $\Psi(S)$ the admissible
homeomorphism associated to an admissible generating function $S$.

\begin{prop}\label{proposition FG->hameo}Let $t\mapsto S_t$, $t\in[0,\delta]$ be a $C^1$ path of super-admissible generating functions,
associated to a fixed Weinstein chart, which is the $C^1$-limit of a smooth path of smooth admissible generating
functions. Then, the path $t\mapsto\Psi(S_t)$ is a $C^0$-Hamiltonian isotopy.
\end{prop}

\begin{remark}We can construct examples of such paths using Darboux coordinates (see Section \ref{section examples} below). By the way,
this proposition gives new examples of $C^0$-Hamiltonian isotopies. As an example, the argument shows that any
Lipschitz continuous path in $G_1$ is a $C^0$-Hamiltonian isotopy.
\end{remark}

To prove Proposition \ref{proposition FG->hameo}, we will need two (classical) lemmas.

\begin{lemma}\label{lemme homeo-fnct}Let $j:U\to V$ be a Weinstein chart. For any integer $k\geq0$, the map $\Psi$ is a homeomorphism between the set of $C^{k+1}$ admissible
generating functions associated to $j$ (endowed with the $C^{k+1}$-topology) and the set of $C^k$ admissible
(diffeo)homeomorphisms associated to $j$ (endowed with the $C^k$-topology).
\end{lemma}

\begin{demo} The details of the proof of this lemma will be left to the reader. We just give here the idea: as in the previous
section, we use the relation between $S$ and $\Psi(S)$. Denote $a=j\circ(\Id\times\Psi(S))\circ j^{-1}$. Then by
construction, $\pi\circ a$ is invertible and one has
$$dS=a\circ(\pi\circ a)^{-1}.$$
This gives continuity properties of $\Psi^{-1}$.

Conversely, if we consider $p_1:M\times M\to M$ the projection on the first factor, and denote $b=j^{-1}\circ dS \circ
j$, then by construction, $p_1\circ b$ is invertible and one has $$\Id\times\Psi(S)=b\circ(p_1\circ b)^{-1}.$$ This
allows to prove continuity properties for $\Psi$.
\end{demo}

\begin{lemma}\label{lemme hamilton-jacobi}Let $t\mapsto S_t$ be a smooth path of smooth admissible generating functions associated
to a fixed Weinstein chart and denote $H$ the compactly
supported Hamiltonian function that generates the Hamiltonian isotopy $t\mapsto\Psi(S_t)$. Then,
$$H(t,x)=-\derive{S_t}{t}(\pi\circ j\circ(\Psi(S_t)^{-1}\times\Id)\circ j^{-1}(x)).$$
\end{lemma}
In $\RN$, this formula is just the classical Hamilton-Jacobi Equation.

\begin{demo} We set $f_t=\Id\times\Psi(S_t)$, $q_t=\pi\circ j\circ f_t\circ j^{-1}$, and denote by $\lambda_0$ the canonical
Liouville form on $T^*M$. We have seen in the proof of Lemma \ref{lemme homeo-fnct} that $dS_t\circ q_t\circ j=j\circ
f_t$.

We first pull back the Liouville form.  Since $\sigma^*\lambda_0=\sigma$ for any 1-form $\sigma$ on $M$,
$j^*q_t^*dS_t^*\lambda_0=j^*q_t^*dS_t=d(S_t\circ q_t\circ j)$. We thus have:
$$d(S_t\circ q_t\circ j)=f_t^*(j^*\lambda_0).$$
We then take derivative with respect to $t$: \begin{eqnarray*}\lefteqn{d\left(\derive{S_t}{t}(q_t\circ j)+dS_t(q_t\circ
j)\cdot\frac{dq_t}{dt}\circ j \right)} \\& = &
f_t^*(\iota_{\frac{df_t}{dt}\circ{f_t^{-1}}}d(j^*\lambda_0))+d(f_t^*(\iota_{\frac{df_t}{dt}\circ{f_t^{-1}}}(j^*\lambda_0))).\end{eqnarray*}
But since $j$ is symplectic, $d(j^*\lambda_0)=\omega\oplus(-\omega)$ hence
$$f_t^*(\iota_{\frac{df_t}{dt}\circ{f_t^{-1}}}d(j^*\lambda_0))=0-\Psi(S_t)^*(\iota_{\frac{\Psi(S_t)}{dt}\circ\Psi(S_t)^{-1}}\omega)=-d(H_t\circ\Psi(S_t)).$$
Therefore, after taking (compactly supported) primitive, we get:
$$\derive{S_t}{t}(q_t\circ j)+dS_t(q_t\circ j)\cdot\frac{dq_t}{dt}\circ j = -H_t\circ\Psi(S_t)+f_t^*(\iota_{\frac{df_t}{dt}\circ{f_t^{-1}}}(j^*\lambda_0)).$$
It remains to show that $dS_t(q_t\circ j)\cdot\frac{dq_t}{dt}\circ j=f_t^*(\iota_{\frac{df_t}{dt}\circ
f_t^{-1}}(j^*\lambda_0))$. To see this, recall that for any one form $\sigma$ on $M$, the pullback $\pi^*\sigma$ by the
canonical projection coincides with $\lambda_0$ on the image of $\sigma$ (which is a smooth submanifold of $T^*M$).
Then,
\begin{align*}dS_t(q_t\circ
j)\cdot\frac{dq_t}{dt}\circ j &= (\pi^*dS_t)(j\circ f_t)\cdot dj\frac{df_t}{dt}\\
&=(j^*\lambda_0)(f_t)\cdot\frac{df_t}{dt}.\end{align*} This concludes our proof.
\end{demo}

\begin{demo}[of Proposition \ref{proposition FG->hameo}]
Let $(S_t)$ be our path of generating functions. By assumption, there is a sequence of smooth paths of smooth
admissible generating functions $(S_t^k)$ that converges in the $C^1$-sense to $(S_t)$. Let $H_k$ be the generating
Hamiltonian function of the Hamiltonian isotopy $\Psi(S_t^k)$.

By Lemma \ref{lemme homeo-fnct}, the isotopies $(\phi_{H_k}^t)=(\Psi(S_t^k))$ $C^0$-converge to $\Psi(S_t)$. Moreover,
by Lemma \ref{lemme hamilton-jacobi}, the Hamiltonian functions $$H_k=\derive{S_t^k}{t}(\pi\circ
j\circ(\Psi(S_t^k)^{-1}\times\Id)\circ j^{-1}(x))$$ also $C^0$-converge. This shows that $(\Psi(S_t))$ is a
$C^0$-Hamiltonian isotopy.
\end{demo}

\section{Examples in $\RN$}\label{section examples}

In this section, we give some examples of elements in $G_2$ and $G_3$ in $\RN$. Using local Darboux coordinates, they
can of course be implanted in other symplectic manifolds.

\subsection{Examples of elements in $G_2$}

In $\RN\times\RN$, there exists globally defined Weinstein charts sending the diagonal to the zero section in $T^*\RN$.
We will use the following one:
$$j:\RN\times\RN\to T^*\RN=\RN\times\RN, (x,y;\xi,\eta)\mapsto(x,\eta;y-\eta,\xi-x).$$

In this Weinstein chart, admissible homeomorphisms and admissible generating functions are associated by the following
relation:
$$f(x,y)=(\xi,\eta)\quad\Longleftrightarrow\quad\begin{cases}\xi=x+\derive{S}{\eta}(x,\eta)\\
y=\eta+\derive{S}{x}(x,\eta)
\end{cases}.$$

Therefore, admissible generating function are the compactly supported $C^1$ functions $S:\RN\to\R$ such that
\begin{itemize}
  \item for all $\eta\in\R^n$, the map $x\mapsto
  x+\derive{S}{\eta}(x,\eta)$ is a homeomorphism of $\R^n$,
  \item for all $x\in\R^n$, the map
  $\eta\mapsto\eta+\derive{S}{x}(x,\eta)$ is a homeomorphism of $\R^n$.
\end{itemize}

\begin{prop}\label{proposition example R2n}
Any compactly supported $C^1$ function $S:\RN\to\R$ such that, in any point $(x,\eta)\in\RN$ the maps
\begin{equation*}x_i\mapsto  x_i+\derive{S}{\eta_i}(x,\eta)\text{ and
}\eta_i\mapsto\eta_i+\derive{S}{x_i}(x,\eta),\text{ for } i\in\{1,\ldots,n\},\end{equation*} are increasing
homeomorphisms of $\R$, is a super-admissible generating function.
\end{prop}

\begin{demo}First, such a function is admissible: for any $x,\eta\in\R^n$ the maps $\eta\mapsto \eta+\derive{S}{x}(x,\eta)$
and $x\mapsto x+\derive{S}{\eta}(x,\eta)$ are homeomorphisms of $\R^n$.

Indeed, one see easily that $\eta\mapsto \eta+\derive{S}{x}(x,\eta)$ is continuous and injective. Since it is compactly
supported, it is also proper and hence is an embedding. Finally, this implies that it is onto, because otherwise its
image would contain non-contractible spheres $\S_{n-1}$. The same argument holds for $x\mapsto
x+\derive{S}{\eta}(x,\eta)$.

Let us now show that $S$ can be approximated in the $C^1$-sense by smooth generating functions.

Let $\chi$ be a smooth non-negative function, defined on $\R^{2n}$, whose support is contained in a disk centered in 0
and with integral equal to 1. For any positive integer $k$, we set $\chi_k=k^{2n}\chi(\frac{\cdot}k)$. Then, it is well
known that the sequence of smooth functions $(S_k)$ defined by
$$S_k(x,\eta)=\chi_k*S(x,\eta)=\int_{\R^{2n}}S(x-u,\eta-v)\chi_k(u,v)\,du\,dv,$$
$C^1$-converges to $S$ as $k$ goes to infinity. Moreover, there exists a compact set that contains the supports of
every $S_k$.

\medskip
let us now prove that the $S_k$ are admissible generating functions. Set
$$\alpha(x,\eta)=x+\derive{S}{\eta}(x,\eta)\text{ and
}\beta(x,\eta)=\eta+\derive{S}{x}(x,\eta).$$ According to the first part of the proof, it is enough to prove that for
any indices $i$, the maps $x_i\mapsto q_i\circ(\chi_k\ast\alpha(x,\eta))$ and $\eta_i\mapsto
p_i\circ(\chi_k\ast\beta(x,\eta))$ are increasing homeomorphisms of $\R$. They are clearly continuous. Since they are
compactly supported, we only need to show that they are increasing. Let us prove it for $x_1\mapsto
q_1\circ(\chi_k\ast\alpha(x,\eta))$. The proof is similar for the others.

Fix $\eta,x_2,\ldots,x_n$ and $x_1<x_1'$ and denote $x=(x_1,x_2,\ldots,x_n)$ and $x'=(x_1',x_2,\ldots,x_n)$. We want to
compare $q_1\circ(\chi_k\ast\alpha(x,\eta))$ with $q_1\circ(\chi_k\ast\alpha(x',\eta))$. By assumption, for all
$(u,v)\in\R^{2n}$,
$$q_1\!\circ\alpha(x-u,\eta-v)\,<\,q_1\!\circ\alpha(x'-u,\eta-v),$$
thus the following integral is non-negative:
$$\int_{R^{2n}}\chi_k(u,v)\left[\,q_1\!\!\circ\!\alpha(x'\!\!-\!u,\eta\!-\!v)-q_1\!\!\circ\!\alpha(x\!-\!u,\eta\!-\!v)\,\right]\,du\,dv.$$
It is moreover positive because it is the integral of a non-negative continuous function which is non-identically zero.
This integral is nothing but $q_1\circ(\chi_k\ast\alpha(x,\eta))-q_1\circ(\chi_k\ast\alpha(x',\eta))$. Therefore the
map $x_1\mapsto q_1\circ(\chi_k\ast\alpha(x,\eta))$ is an increasing homeomorphism of $\R$.
\end{demo}

\begin{remark}The conjugation $\mu_t\circ\phi\circ\mu_t^{-1}$ by the Liouville flow $\mu_t:x\mapsto e^{t/2}x$ of an homeomorphism $\phi$ associated to a
generating function $S$ like in Proposition \ref{proposition example R2n}, is also admissible and is associated to the
generating function $e^tS(e^{-t/2}x,e^{-t/2}\eta)$. This function satisfies the hypothesis of Proposition
\ref{proposition example R2n} and hence is also a super-admissible generating function. It follows that such a $\phi$
is in $G_2$.
\end{remark}

\begin{remark}Any generating function like in Proposition \ref{proposition example R2n}, which is \textbf{not}
$C^{1,1}$, gives rise to an example of element which is in $G_2$ but not in $G_1$.
\end{remark}

\subsection{Fibered rotations in $\R^2$}

By definition a \textit{fibered rotation} is an homeomorphism $\phi$ of $\R^2$ described in polar coordinates
$(r,\theta)$ by the formula
$$\phi(r,\theta)=(r,\theta+\rho(r)),$$
for some continuous \textit{angular} function $\rho:(0,+\infty)\to\R$ with bounded support. It is easily checked that
any fibered rotation lies in the identity component of the group of compactly supported area preserving homeomorphism
of $\R^2$.

We consider $\mu_t$ the Liouville flow given by $\mu_t(r,\theta)=(e^{t/2}r,\theta)$. Its commutator with a fibered
rotation is given by
$$[\mu_t,\phi](r,\theta)=(r,\theta-\rho(r)+\rho(e^{-t/2}r)).$$
If $\phi$ is moreover a diffeomorphism, the generating Hamiltonian of the isotopy $t\mapsto[\mu_t,\phi]$ is
$$H(t,r,\theta)=r\rho(e^{-t/2}r)-\frac12\int_0^r\rho(e^{-t/2}s)\,ds.$$

Now suppose that $\rho$ is a continuous and integrable angular function, such that $r\rho(r)$ converges to 0 when $r$
tends to 0. Suppose also that $\rho_k$ is a sequence of smooth compactly supported angular functions (in particular
they vanish nearby 0) that converges uniformly to $\rho$ on any compact subset of $(0,+\infty)$. Then, clearly, the
associated sequence of fibered rotations $(\phi_k)$ converges in the $C^0$-sense to $\phi$, and the sequence of
Hamiltonians $(H_k)$ generating the isotopies $t\mapsto[\mu_t,\phi_k]$ also $C^0$-converges.

As a consequence, \emph{any fibered rotation associated to an {integrable} angular function $\rho$ such that
$r\rho(r)\stackrel{r\to 0}{\longrightarrow}0$, belongs to $G_3$.}

\begin{remark}This gives examples of elements that are in $G_3$ but not in $G_2$: if $\rho$ is not finite (nearby 0),
the fibered rotation $\phi$ cannot be in $G_2$. Indeed, the angle between a vector and its image by an admissible
homeomorphism is bounded by $\pi$. Therefore, this angle has to be finite for elements of $G_2$.
\end{remark}

\subsection*{Aknowledgments}
I wish to thank the members of the ANR project "Symplexe" for all they taught me and for the very motivating working
atmosphere during our meetings. In particular, I thank Frédéric Le Roux and Pierre Py for many interesting discussions
on the subject of the present article, and for their comments on its first version. I am also grateful to Claude
Viterbo for his constant support.


\begin{thebibliography}{10}

\bibitem{banyaga}
Augustin Banyaga.
\newblock {\em The structure of classical diffeomorphism groups}, volume 400 of
  {\em Mathematics and its Applications}.
\newblock Kluwer Academic Publishers Group, Dordrecht, 1997.

\bibitem{bounemoura}
Abed Bounemoura.
\newblock {\em Simplicit{\'e} des groupes de transformations de surfaces}.
\newblock Ensaios Matematicos, 2008.

\bibitem{calabi}
Eugenio Calabi.
\newblock On the group of automorphisms of a symplectic manifold.
\newblock In {\em Problems in analysis ({L}ectures at the {S}ympos. in honor of
  {S}alomon {B}ochner, {P}rinceton {U}niv., {P}rinceton, {N}.{J}., 1969)},
  pages 1--26. Princeton Univ. Press, Princeton, N.J., 1970.

\bibitem{fathi}
Albert Fathi.
\newblock Structure of the group of homeomorphisms preserving a good measure on
  a compact manifold.
\newblock {\em Ann. Sci. \'Ecole Norm. Sup. (4)}, 13(1):45--93, 1980.

\bibitem{gamb-ghys}
Jean-Marc Gambaudo and Etienne Ghys.
\newblock Enlacements assymptotiques.
\newblock {\em Topology}, 36(6):1355--1379, 1997.

\bibitem{haissinsky}
Peter Ha{\"{\i}}ssinsky.
\newblock L'invariant de {C}alabi pour les hom\'eomorphismes quasiconformes du
  disque.
\newblock {\em C. R. Math. Acad. Sci. Paris}, 334(8):635--638, 2002.

\bibitem{humi}
Vincent Humiliere.
\newblock On some completions of the space of hamiltonian maps.
\newblock {\em Bull. Soc. Math. France}, 136, fascicule 3 (2008), 373-404.

\bibitem{leroux}
Frédéric Le~Roux.
\newblock Simplicity of the group of area preserving homeomorphisms of the disc
  and fragmentation of symplectic diffeomorphisms, 2008.
\newblock preprint.

\bibitem{Oh}
Yong-Geun Oh.
\newblock The group of hamiltonian homeomorphisms and continuous hamiltonian
  flows, 2006.
\newblock arXiv.org:math/0601200.

\bibitem{muller-oh}
Yong-Geun Oh and Stefan Muller.
\newblock The group of {H}amiltonian homeomorphisms and ${C}\sp 0$ symplectic
  topology, 2004.

\bibitem{viterbo}
Claude Viterbo.
\newblock Erratum to: ``{O}n the uniqueness of generating {H}amiltonian for
  continuous limits of {H}amiltonians flows'' [{I}nt. {M}ath. {R}es. {N}ot.
  {\bf 2006}, {A}rt. {ID} 34028, 9 pp.; mr2233715].
\newblock {\em Int. Math. Res. Not.}, pages Art. ID 38784, 4, 2006.

\bibitem{Weinstein}
Alan Weinstein.
\newblock Symplectic manifolds and their {L}agrangian submanifolds.
\newblock {\em Advances in Math.}, 6:329--346 (1971), 1971.

\end{thebibliography}
\end{document}